\documentclass[psamsfont,12pt]{amsart}

\usepackage{amsmath,amsfonts,amssymb,graphics,graphicx,latexsym,amscd,subfigure,hyperref}
\usepackage[usenames]{xcolor}
\usepackage[retainorgcmds]{IEEEtrantools}
\usepackage{graphicx}
\usepackage{amssymb}
\usepackage{xypic}
\usepackage{hyperref}
\usepackage{amsmath,amssymb}
\usepackage{dsfont}\let\mathbb\mathds
\usepackage[english]{babel}

\newtheorem{theorem}{Theorem}[section]
\newtheorem{lemma}[theorem]{Lemma}
\newtheorem{corollary}[theorem]{Corollary}
\newtheorem{proposition}[theorem]{Proposition}

\theoremstyle{remark}
\newtheorem{rem}[theorem]{Remark}

\theoremstyle{definition}
\newtheorem{definition}[theorem]{Definition}

\usepackage{marvosym}

\def\del{\partial}              % 'd rond'

\def\bC{\mathbb C}          % corps des complexes
\def\bR{\mathbb R}          % corps des reels
\def\bQ{\mathbb Q}          % corps des rationnels
\def\bN{\mathbb N}          % groupe des entiers naturels
\def\bZ{\mathbb Z}          % anneau des entiers relatifs
\def\bT{\mathbb T}          % tore
\def\bP{\mathbb P}

\def\bcp{\mathbb C \mathbb P}

\def\mS{\mathcal{S}}            % pour le S ronde
\def\mS{\mathcal{S}}            % pour le S ronde
\def\mK{\mathcal{K}}

\def\bfB{\mbox{{\bf B}}}

\def\kt{\mathfrak{t}}
\def\Lm{L^{\text{min}}}

\def\b{\alpha}

\def\bi{{\bf{ i}}}
\def\ps{s}
\def\pol{P}
\def\ra{\rightarrow}

\def\vol{d\varpi}

\def\zeta{{\rm A}}
\def\del{\partial}

\def\Lap{\Delta}

\def\Hess{{\rm{ Hess\,}} }

\def\w{{m}} %weight

 \newcommand*{\quot}[2]%
{\ensuremath{%
   \raisebox{.35ex}{\ensuremath{#1}}\big/\raisebox{-.35ex}{\ensuremath{#2}}}}

\begin{document}

\title{Toric aspects of the first eigenvalue}

\author{Eveline Legendre}
\author{Rosa Sena-Dias}
\date{\today}

\address{Eveline Legendre\\ Universit\'e Paul Sabatier\\
Institut de Math\'ematiques de Toulouse\\ 118 route de Narbonne\\
31062 Toulouse\\ France}
\email{eveline.legendre@math.univ-toulouse.fr}

\address{Rosa Sena-Dias\\Centro de An\'alise Matem\'atica, Geometria e Sistemas Din\^amicos\\ Departamento de Matem\'{a}tica, Instituto Superior T\'{e}cnico\\ Av. Rovisco Pais, 1049-001 Lisboa\\ Portugal}
\email{rsenadias@math.ist.utl.pt}

\thanks{RSD was partially supported by FCT/Portugal through projects PEst-OE/EEI/LAOO9/2013, EXCL/MAT-GEO/0222/2012 and PTDC/MAT/117762/2010 and EL is partially supported by the ANR French grant EMARKS. We would also like to thank CAST for a travel grant that allowed EL to visit Lisbon.}
%\subjclass[2000]{Primary 35J25; Secondary 35P15}
%% voir classification http://www.ams.org/msc/
%\keywords{toric K\"ahler orbifold, K\"ahler-Eintein metrics,  Laplacian, eigenvalues}

\begin{abstract} 
In this paper we study the smallest non-zero eigenvalue $\lambda_1$ of the Laplacian on toric K\"ahler manifolds. We find an explicit upper bound for $\lambda_1$ in terms of moment polytope data. We show that this bound can only be attained for $\bC\bP^n$ endowed with the Fubini-Study metric and therefore  $\bC\bP^n$ endowed with the Fubini-Study metric is spectrally determined among all toric K\"ahler metrics.  We also study the equivariant counterpart of $\lambda_1$ which we denote by $\lambda_1^T$. It is the smallest non-zero eigenvalue of the Laplacian restricted to torus-invariant functions. We prove that $\lambda_1^T$ is not bounded among toric K\"ahler metrics thus generalizing a result of Abreu-Freitas on $S^2$. In particular, $\lambda_1^T$ and $\lambda_1$ do not coincide in general.
\end{abstract}
\maketitle
%\tableofcontents
\section{Introduction}
Toric K\"ahler manifolds are very symmetric K\"ahler manifolds for which there is a concrete parametrization of the space of K\"ahler metrics.  More concretely, they are symplectic manifolds admitting an {effective} Hamiltonian action from a maximal torus and endowed with a compatible, torus invariant Riemannian metric giving rise to an integrable complex structure. The underlying symplectic manifold is completely characterized by a combinatorial object which is a convex polytope called moment polytope arising as the image of the moment map for the torus action. Toric K\"ahler metrics are parametrized by convex functions on that moment polytope satisfying certain properties as we will discuss in section \ref{background}. Toric K\"ahler manifolds have played a crucial role in studying important questions in geometry. {Mabuchi was one the first to study their K\"ahler Geometry in \cite{mab}}. In \cite{don:scalar}, Donaldson was able to fully characterize those toric K\"ahler surfaces 
which admit constant scalar curvature thus settling an important conjecture in K\"ahler geometry in the toric 
context for real dimension 4. 
There has been a lot of interest in studying toric spectral geometry and, in particular, inverse spectral questions in this toric context as well (see \cite{abreufreitas}, \cite{DGS}). 

Given a Riemannian manifold $(M,g)$, the Riemannian metric determines a Beltrami-Laplace operator whose smallest non-zero eigenvalue, which we also refer to as the {\it first eigenvalue}, and is denoted by $\lambda_1(g)$, carries a surprising amount of geometric information. There has been great deal of effort put into finding sharp bounds for $\lambda_1$ with geometric meaning (see \cite{BGM}). In \cite{hersch}, Hersch discovered an upper bound for $\lambda_1$ for metrics on $S^2$. Bourguignon--Li--Yau found an upper bound for $\lambda_1$ for K\"ahler manifolds endowed with a full holomorphic embedding into projective space~\cite{BLY}. This result has been extended in some ways to K\"ahler manifolds carrying Gieseker stable bundle~\cite{AGL} by allowing maps to Hermitian symmetric spaces~\cite{BiGh}. Polterovich (see \cite{leonidP}) looked at boundedness of $\lambda_1$ in the context of symplectic manifolds. He showed, in 
particular, that there are symplectic manifolds admitting compatible Riemannian metrics whose $\lambda_1$ is arbitrarily large. One of the questions we want to address here is: ``are there geometric bounds on $\lambda_1(M,g)$ where $M$ is a toric manifold and $g$ is a toric K\"ahler metric on it?". As it turns out, one can always use Bourguignon--Li--Yau's result in the toric context, and we use it to give an 
explicit bound for $\lambda_1$ in terms of moment polytope data. More precisely,  we prove the following theorem.
\begin{theorem}\label{thm_bound_lambda_1}
Let $(M^{2n},\omega)$ be a toric symplectic manifold endowed with a toric K\"ahler structure whose Riemannian metric we denote by $g$. Let $\pol\subset \bR^n$ be its moment polytope. There is an integer, $k_0(\pol)\geq 1$ such that for any $k\geq k_0(\pol)$ 
$$
\lambda_1(g)\leq \frac{2nk(N_k+1)}{N_k},
$$
where $N_k+1=\sharp(P\cap \bZ^n/k)$. If $\pol$ is integral (i.e its vertices lie in $\bZ^n$), then we have a finer bound given by
$$
\lambda_1(g)\leq \frac{2n(N+1)}{N},
$$
where $N+1=\sharp(P\cap \bZ^n)$ is the number of integer points in $\pol$.
\end{theorem}
We will make $k_0(\pol)$ explicit ahead. The Fubini-Study metric realizes the bound in the above theorem. In fact we show that this is the {\it only} toric K\"ahler metric that does saturate this bound in the integral case. 
\begin{theorem}\label{saturate}
Let $(M^{2n},\omega)$ be an integral toric symplectic manifold endowed with a toric K\"ahler structure whose Riemannian metric we denote by $g$. Let $N+1$ be the number of integer points in the moment polytope of $M$. If
$$
\lambda_1(g)=\frac{2n(N+1)}{N},
$$
then $M$ is equivariantly symplectomorphic to $\bC\bP^n$ and this symplectomorphism takes $g$ into the Fubini-Study metric on $\bC\bP^n$.
\end{theorem}
It was previously known (see \cite{BGM}) that the Fubini-Study metric on $\bC\bP^n$ is determined by the spectrum among all K\"ahler metrics on $\bC\bP^n$ compatible with the standard complex structure. It was also proved by Tanno (see \cite{Tanno}) that, if a K\"ahler manifold of real dimension less than $12$ has the same spectrum as $\bC\bP^n$ with the Fubini-Study metric, then it is holomorphically isometric to it. A simple consequence of the above theorem is that the spectrum of the Laplacian of a toric K\"ahler metric on an integral toric manifold determines if the manifold is $\bC\bP^n$ endowed with the Fubini-Study metric.

\begin{corollary}
An integral toric K\"ahler manifolds which has the same spectrum as $(\bC\bP^n,\omega_{FS},J_0)$ is holomorphically isometric to it. 
\end{corollary}

Another interesting question is that of spectrally characterizing either constant scalar-curvature, extremal or K\"ahler-Einstein toric K\"ahler metrics. In \cite{DGS} the authors prove that the equivariant spectrum determines if a toric K\"ahler metric has constant scalar curvature. A variation of the argument there would show that the equivariant spectrum also determines if a metric is extremal.

Going back to the first eigenvalue, there are various bounds that one can write down for toric K\"ahler manifolds using Bourguignon--Li--Yau's bound, see~\S\ref{BLYboundsub}. It would be interesting to see what the best bound is for a given toric manifold, once we fix the polytope. In particular, one could hope to improve the bound in Theorem  \ref{thm_bound_lambda_1} for special classes of manifolds (monotone, Fano..) or special classes of metrics say extremal toric K\"ahler metric, or K\"ahler-Einstein metrics. In \cite{AJK} the authors prove that a toric K\"ahler--Einstein manifold whose connected component of automorphism group is a torus is never {\it $\lambda_1$--extremal}, where $\lambda_1$--extremal means extremal for the first eigenvalue with respect to local variations in the K\"ahler metrics space. Hence, in general, we cannot expect a toric K\"ahler--Einstein metric to saturate fine bounds. Another natural candidate to consider is a balanced metric when it exists, see discussion~\S\ref{BLYboundsub}.

However, the K\"ahler--Einstein property is somewhat reflected in the first eigenvalue. In fact, one can prove an improvement and a converse of Matsushima Theorem~\cite{matsushima} in~\S\ref{moment_map_eigenfunction}. We show that a toric K\"ahler metric is K\"ahler-Einstein if and only if the coordinates of its moment map are eigenfunctions for $\lambda_1$. 
\begin{proposition}\label{propo:moment_map_eigenfunction}\footnote{It is possible that this result was previously known but the authors did not find a reference for it in the literature and thus state it and prove it.} \label{propKEeigenfunction} Let $(M,\omega, T)$ be a compact symplectic toric orbifold with moment map $x:M\ra \kt^*$. Then $(M,g,J,\omega, T)$ is a K\"ahler--Einstein toric {orbifold} with Einstein constant $\lambda$ if and only if, up to an additive constant, the moment map satisfies
 \begin{equation}\label{toricEIGEN}
 2\lambda\langle x,b\rangle=\Delta^{g} \langle  x,b\rangle \;\;\;\;\;\;  \forall b\in \kt.\end{equation} 
In this case, $2\lambda$ is the smallest non-vanishing eigenvalue for the  K\"ahler--Einstein orbifold toric metric.
\end{proposition}

Matsushima's theorem implies that a necessary condition for a toric K\"ahler metric to be K\"ahler-Einstein is that its $\lambda_1$ be a multiple eigenvalue with multiplicity at least equal to half the dimension of the manifold. What's more, it follows from the above proposition that one can see if a metric is K\"ahler- Einstein by simply checking if its moment map coordinates are eigenfunctions for $2\lambda$.

On a toric manifold endowed with a  torus invariant metric one can consider a toric version of $\lambda_1$ namely $\lambda_1^T$ defined to be the smallest non-zero invariant eigenvalue of the Laplacian i.e. the smallest eigenvalue of the Laplacian restricted to torus invariant functions. We clearly have $\lambda_1\leq \lambda_1^T$. In \cite{abreufreitas}, Abreu--Freitas studied $\lambda_1^T$ for the simplest toric manifold, namely $S^2$ with the usual $S^1$ action by rotations around an axis. They proved it was unbounded (both above and below) among $S^1$-invariant metrics. In this paper we generalize their results, by using an original approach for the upper bound, on all toric manifolds. We are able to prove the following.
\begin{theorem} \label{THEOminLAMBDA1} 
Let $(M,\omega, T)$ be a compact symplectic toric orbifold, let $\mK_{\omega}^T$ be the set of all toric K\"ahler metrics on $(M,\omega, T)$. Then,
$$\inf\{\lambda_1^T(g)\,|\, g\in\mK_{\omega}^T\}=0.$$ and
$$\sup\{\lambda_1^T(g)\,|\, g\in\mK_{\omega}^T\}=+\infty.$$
\end{theorem}
Combining Theorem~\ref{thm_bound_lambda_1} and \ref{THEOminLAMBDA1}, we see that there are toric K\"ahler manifolds for which $\lambda_1$ does not coincide with $\lambda_1^T$. For toric K\"ahler--Einstein metrics, it follows from Matsushima Theorem~\cite{matsushima} that $\lambda_1=\lambda_1^T$ as there are invariant eigenfunctions for $\lambda_1$. It would be interesting to characterize those toric K\"ahler manifolds for which this occurs. Given a weight vector $\w\in \bZ^n$, one could also define $\w$-equivariant $\lambda_1$ which we denote by $\lambda_1^\w$ as the lower non-vanishing eigenvalue of the Laplacian restricted to the set of $\w$-equivariant functions
$$
\{f\in\mathcal{C}^\infty(M,\bC):f(e^{\bi\theta}p)=e^{{\bi}\theta\cdot \w}f(p), \quad \forall p\in M,\, \theta \in \bR^n\}.
$$
One could prove a similar result in this setting and again it would be interesting to understand which metrics have $\lambda_1=\lambda_1^\w$ and how this depends on $\w$. Note that $\lambda_1^T=\lambda_1^0$. Recently, in \cite{hm}, Hall-Murphy proved that on any toric manifolds $\lambda_1^T$ restricted to the class of toric K\"ahler metrics whose scalar curvature is non-negative is bounded and this generalizes another result in \cite{abreufreitas}.

The paper is organized as follows. In section \ref{background} we quickly review some basic facts about toric manifolds and their toric K\"ahler metrics. The reader is encouraged to consult the references for more details and proofs. We also give a proof of  Proposition \ref{propKEeigenfunction}.  In section \ref{invariant_lambda_1} we study $\lambda_1^T$ and generalize Abreu-Freitas' result to prove Theorem \ref{THEOminLAMBDA1}. Section \ref{bound_lambda_1} deals with $\lambda_1$ and there we prove Theorems \ref{thm_bound_lambda_1} and \ref{saturate}.

\noindent \textbf{Acknowledgements.} The authors would like to thank Emily Dryden and Julien Keller for interesting conversations concerning the topic of this paper and also Stuart Hall and Tommy Murphy for sharing their preprint.

\section{Background}\label{background}
\subsection{Toric K\"ahler geometry}
This section does not contain all the ingredients of symplectic toric geometry needed in subsequent sections, we only lay down the notation and refer to the classical references for this theory (in particular for proofs of what is claimed in this section) like \cite{abreu,abreuOrbifold,delzant:corres,don:estimate,guillMET,convexMoment,LT:orbiToric}.

Let $(M^{2n},\omega,T^n)$ be a compact toric symplectic orbifold. It admits a moment map $x : M \ra \pol\subset \kt^*$ where $\kt= \mbox{Lie } T$ is the Lie algebra of $T$ and $\kt^*$ is its dual such that for all $a\in\kt$ $$-d\langle x,a\rangle = \omega(X_a,\cdot)$$ where $X_a$ is the vector field on $M$ induced by the $1$--parameter subgroup associated to $a$. The image of $x$, that we denote $\pol$, is called the moment polytope. It is a convex simple (i.e. its vertex are the intersection of exactly $n$--facets) polytope in $\kt^*$.

\begin{definition} Consider $P\subset \kt^*$ a simple polytope, $\nu=\{\nu_1,\dots, \nu_d\}$ a set of vectors in $\kt$ which are normal to the facets of $\pol$ and inward pointing. Let $\Lambda$ be the lattice in $\kt$ such that $T=\kt/\Lambda$.
If $\nu\subset \Lambda$, the triple $(\pol, \nu,\Lambda)$ is called a {\bf labelled rational}\footnote{To recover the original convention introduced by Lerman and Tolman in the rational case, take $m_k\in \bZ$ such that $\frac{1}{m_k}\nu_
k$ is primitive in $\Lambda$ so $(\pol, m_1,\dots m_d,\Lambda)$ is a rational labelled polytope.} polytope. If each subset of vectors in $\nu$, normals to facets meeting at a vertex, forms a basis of $\Lambda$, then we say that $(\pol, \nu,\Lambda)$ is {\bf Delzant}.\end{definition} 

The Delzant--Lerman--Tolman correspondence states that compact toric symplectic orbifolds are in one to one correspondence with rational labelled polytopes and are smooth if and only if the rational labelled polytopes is Delzant. 

In this text, we often identify $\kt$ with $\bR^n$ and $\Lambda$ with $\bZ^n$.

\begin{definition}
Let $(\pol, \nu)$ be a labelled polytope. The functions $L_1,\dots, L_d\in \mbox{Aff}(\kt^*, \bR)$ are said to be the defining functions of $(\pol, \nu)$ if $\pol= \{ x\in\kt^*\,|\,L_k(x)\geq0\}$ and $\frac{dL_k}{dx}=\nu_k$.
\end{definition}
Let $\mathring{\pol}$ denote the interior of $\pol$. On the pre-image of the interior of the polytope $\mathring{M}=x^{-1}(\mathring{\pol})$, the action of $T$ is free. The action--angle coordinates $(x,\theta)=(x_1,\dots, x_n,\theta_1\dots, \theta_n)$ are local coordinates on $\mathring{M}$ used to (locally) identify $\mathring{M}$ with $\mathring{P}\times T$ where the first projection coincides with the moment map and 
\begin{equation}\label{eq:OMEGAinaacoord}
\omega=\sum_{i=1}^ndx_i\wedge d\theta_i. 
\end{equation}

As it is shown in \cite{a2}, the space of compatible $T$--invariant K\"ahler metrics on $(M,\omega,T)$ is parametrized by the set of {\it symplectic potentials} which is denoted by $\mS(\pol,\nu)$ (up to the addition of an affine linear function). The set $\mS(\pol,\nu)$ is defined as the subset of functions $u\in C^{\infty}(\mathring{\pol},\bR)\cap C^{0}(\pol,\bR)$, such that 
 \begin{itemize}
   \item[(i)] $u-\frac{1}{2}\sum_{k=1}^d L_k \log L_k \in C^{\infty}({\pol},\bR)$;\\
   \item[(ii)] the restriction of $u$ to $\mathring{\pol}$ is strictly convex; \\
   \item[(iii)] for each face $F$ of $\pol$, the restriction of $u$ to $\mathring{F}$ (the relative interior of $F$) is strictly convex.
 \end{itemize}

\begin{definition}\label{guilleminPOTENTIAL} The Guillemin potential $u_o \in \mS(\pol,\nu)$ is 
\begin{equation} u_o = \frac{1}{2} \sum_{i=1}^d (L_k \log L_k -L_k). \end{equation} It corresponds to the K\"ahler toric metric on $(M,\omega)$ obtained via the K\"ahler reduction of $\bC^d$, see~\cite{guillMET}.
\end{definition}

Given $u\in \mS(\pol,\nu)$, the metric defined by
\begin{equation}\label{ActionAnglemetric} g_u= \sum_{i,j=1}^nu_{ij}dx_i\otimes dx_j +u^{ij}d\theta_i\otimes d\theta_j\end{equation} 
where $u_{ij}=\frac{\del^2 u}{\del x_i\del x_j}$ and $(u^{ij})=(u_{ij})^{-1}$, is a $\kt$--invariant 
 K\"ahler metric on $\mathring{\pol}\times T\simeq \mathring{M}$ compatible with $\omega$. Conditions $(i), (ii),(iii)$ ensure that $g_u$ is the restriction of a smooth metric on $M$. For convenience, we denote $H_{ij}^u=u^{ij}$, $\,G^u_{ij} = u_{ij}$, $\, H^u=(H_{ij}^u)$ and $\, G^u=(G_{ij}^u)$. One can prove that any toric K\"ahler structure on $(M,\omega)$ can be written using a symplectic potential in $\mS(\pol,\nu)$ as above ({see \cite{a2}}). 

{\begin{rem}\label{remCPLXcoor} Given $u \in \mS(\pol,\nu)$ the map ${\frac{\partial u}{\partial x}} : \mathring{P} \longrightarrow \kt$ is a diffeomorphism (because $u$ is strictly convex) and the coordinates $z= y +\bi \theta$, where $y={\frac{\partial u}{\partial x}}$, are local complex coordinates on $\mathring{M}$. See \cite[\S A1.3]{GuMM}) for instance.
\end{rem}}

 {Note that in \cite{guillMET}, $u_o$ is defined to be $u_o = \frac{1}{2} \sum_{i=1}^d (L_k \log L_k )$. This will yield the same metric as the metric we define via formula \ref{ActionAnglemetric} and complex coordinates which are related to the ones in \cite{guillMET} by an overall translation.}

Abreu~\cite{abreu} computed the curvature of a compatible K\"ahler toric metric, $g_{u}$, in terms of its symplectic potential $u$. The scalar curvature of $g_u$ is the pull-back by $x$ of the function \begin{equation}\label{abreuForm} \mbox{scal}_{u}=-\sum_{i,j=1}^n\frac{\del^2 H^u_{ij}}{\del x_i \del x_j}. \end{equation}  Moreover, the Ricci curvature is \begin{equation}
\label{ricciFORMULA} \rho^{g_u} =\frac{-1}{2}\sum_{i,l,k} H^u_{li,ik}dx_k\wedge d\theta_l.\end{equation} {(See for instance \cite{lejmi} where the above formula is proved in the more general context of almost K\"ahler metrics.)}

 \subsection{K\"ahler--Einstein metrics and moment map coordinates as eigenfunctions of the Laplacian}\label{moment_map_eigenfunction}
Let $(M,g_u,J,\omega, T)$ be a compact K\"ahler toric manifold with moment map $x$ and denote by $\Delta^{u}$, the Laplacian with respect to the Riemannian metric $g_u$. Recall that $(M,g_u,J,\omega)$ is K\"ahler--Einstein if there exists $\lambda$ such that $\lambda\omega=\rho^{g_u}$ where $\rho^{g_u}$ is the Ricci form of the Chern and Levi connection. We say that $\lambda$ is the Einstein constant. In the compact toric setting, $\lambda>0$.

Next we proceed to prove Proposition (\ref{propo:moment_map_eigenfunction})
\begin{proof} Expressing the Laplacian (i.e $\Delta^g= -\mbox{Div}^g\mbox{grad}^g$) in the action angle coordinates \eqref{ActionAnglemetric}, we get    
 \begin{equation}\label{toricLAP}\Delta^{u}= - \sum_{i,j=1}^n \left[G_{ij}\frac{\del^2}{\del \theta_i\del \theta_j} + \frac{\del}{\del x_i}\left( H_{ij} \frac{\del}{\del x_j}\right)\right],\end{equation} 
 so that
 \begin{equation}\label{toricBOCHNER}
   d\Delta^{u}\langle x,b\rangle= -\sum_{i,j,k=1}^n  H_{ij,ik}b_j dx_k =-{2}\rho^{g_u}(X_b,\cdot)\;\;\;\;\;\;  \forall b\in \kt\end{equation} 
    using \eqref{ricciFORMULA}. 
   From~\eqref{toricBOCHNER}, we see that $\Delta^{g_u}x$ is a moment map for $2\rho^{g_u}$ and, in the K\"ahler--Einstein case $\lambda\omega=\rho^{g_u}$, this implies that $2\lambda x-\Delta^{g_u}x= \alpha\in \kt^*$ is constant. Thus $x-\frac{\alpha}{2\lambda}$ satisfies~\eqref{toricEIGEN}.
   
   The converse is also a simple computation. Indeed, assuming~\eqref{toricEIGEN}, we have  
 $$\Delta^{g_{u}} x_i = -\sum_{j=1}^n \frac{\partial H_{ij}}{\partial x_j} = 2\lambda x_i$$ for $i=1,\dots n$. Inserting this in~\eqref{ricciFORMULA}, we get \begin{equation}\begin{split}\rho^{g_u}(\cdot,\cdot) &=\frac{-1}{2}\sum_{i,l,k=1}^n H_{li,ik}dx_k\wedge d\theta_l \\
 &= \frac{1}{2}\sum_{l,k=1}^n \frac{\del}{\del x_k}(2\lambda x_l)dx_k\wedge d\theta_l\\
 &=\lambda\sum_{k=1}^n dx_k\wedge d\theta_k =\lambda \omega,\end{split}\end{equation} as in~\eqref{eq:OMEGAinaacoord}.\end{proof}

\section{The first invariant eigenvalue $\lambda_1^T$}\label{invariant_lambda_1}
 \subsection{Minimizing $\lambda_1^T$}

The goal of this subsection is to show the first part of Theorem \ref{THEOminLAMBDA1}. With the notation introduced in Section~\ref{background}, the first part of Theorem~\ref{THEOminLAMBDA1} would follow from
  \begin{equation}\label{minLAMBDA1}
    \inf_{u\in\mS(\pol,\nu)}\{\lambda_1(g_u)\}=0.
  \end{equation}
  An easy computation shows that for any $T$--invariant function $$\int_{M} g_u (\nabla^{g_u} f,\nabla^{g_u} f)dv_{g_u} = \int_{T^n}d\theta_1\wedge\dots\wedge d\theta_n  \int_{\pol} H^u(d f,d f)dx_1\wedge\dots \wedge dx_n.$$ Here $d f$ denotes the differential of $f$ seen as a function on $\pol$. We fix coordinates on $\kt^*$ and, by translating if necessary, we assume that $\int_{\pol} x_i\vol =0$ where we have set $\vol=dx_1\wedge\dots \wedge dx_n$. The Rayleigh characterization of the first eigenvalue tells us that for any $i=1,\dots, n$ \begin{equation}
    \lambda_1(g_u) \leq \frac{\int_{\pol} H^u(d x_i, d x_i)\vol}{\int_{\pol}x_i^2\vol}= \frac{\int_{\pol} u^{ii}\vol}{\int_{\pol}x_i^2\vol}
  \end{equation} with equality if and only if $x_i$ is an eigenfunction of the Laplacian $\Lap^{g_u}$. Since the denominator does not depend on $u$, to show \eqref{minLAMBDA1}, it is sufficient to show that we can find $u\in\mS(\pol,\nu)$ with arbitrarily small $u^{ii}$, as Abreu and Freitas did for $S^1$--invariant metrics on $S^2$ in~\cite{abreufreitas}.\\

 Take any $u_o\in\mS(\pol,\nu)$ and for any positive real number $c>0$ put $u_c= u_o +\frac{c}{2}x_i^2$. First, we will show that $u_c^{ii}$ decreases when $c$ increases. We have $\Hess u_c = \Hess u_o + cE_{i}$ where $E_{i}=(\delta_{li}\delta_{ki})_{1\leq l,k\leq n}$ and $\delta_{li}$ being the Kronecker symbol. In particular, \begin{equation}\label{determinantHessu_c}
   \det \Hess u_c = \det{\Hess} u_o + c\det M_{ii}
 \end{equation} where $M_{lk}$ denotes the $(l,k)$-minor matrix of $\Hess u_o$. Note that $M_{ii}$ is positive definite at each point in $\mathring{\pol}$ since it corresponds to the restriction of the metric $g_{u_o}$ (as a metric on $\mathring{\pol}$) to the orthogonal space to $\frac{\del}{\del x_i}$ with respect to the Euclidean metric. In particular for any $c>0$, formula~\eqref{determinantHessu_c} gives $\det \Hess u_c>0$. Now, since the $(i,i)$-minor matrices of $\Hess u_o$ and $\Hess u_c$ are the same we have
 \begin{equation}\label{decreasinguii}
   u_c^{ii}= \frac{\det M_{ii}}{\det{\Hess} u_o + c\det M_{ii}}
 \end{equation} Thus, $u_c^{ii}\ra0$ when $c\ra +\infty$.\\

 Now, we will show that $u_c\in \mS(\pol,\nu)$ for all $c>0$ by verifying each of the conditions $(i),(ii)$ and $(iii)$ of the definition, see~\S\ref{background} :
 \begin{itemize}
   \item[(i)] $u_c-\frac{1}{2}\sum_{k=1}^dL_k\log L_k = cx_i^2 +(u_o- \frac{1}{2}\sum_{k=1}^dL_k\log L_k) $ is smooth since $u_o \in \mS(\pol,\nu)$;
   \item[(ii)] let $x\in\mathring{\pol}$, $(\Hess u_c)_{x}$ is positive definite because it is the sum of a positive definite matrix namely $\Hess u_o$ with a semi-positive definite matrix. 
   \item[(iii)] let $F$ be a face of $\pol$ and $x\in\mathring{F}$. The restriction of  $(\Hess u_c)_{x}$ to the tangent space to $F$ is again the sum of a positive definite form namely $\Hess {u_o}_{|F}$ with  a semi-positive definite form.
 \end{itemize} Hence $u_c\in \mS(\pol,\nu)$ for all $c>0$ and $\lambda_1(g_{u_c})\longrightarrow 0$ when $c\ra +\infty$. This proves~\eqref{minLAMBDA1}.

 \subsection{Maximizing $\lambda_1^T$}

The goal of this subsection is to show the second part of Theorem \ref{THEOminLAMBDA1}. Let $(\pol,\nu)$ be the labelled moment polytope of a symplectic toric orbifold $(M,\omega,T)$. Without loss of generality, we assume, in this section, that $0\in\mathring{\pol}$. In particular, the defining functions $L_k(x)=\langle x,\nu_k\rangle + c_k$ satisfy $L_k(0)=c_k >0$. Let $u_o\in \mS(\pol,\nu)$ be the Guillemin potential, that is, 
$$u_o=\frac{1}{2}\sum_{i=1}^d L_i \log L_i-L_i$$ 
and $G_o= \mbox{Hess }u_o$ and $H_o=(\mbox{Hess }u_o)^{-1}$. Choosing coordinates and an inner product $(\cdot,\cdot)$, we see $G_o$ and $H_o$ as matrices. For $\ps>1$, we denote $u_o^\ps$, the Guillemin potential of $\ps\pol$ which is the dilation of $\pol$ by an $\ps$-factor. The defining affine-linear functions of $\ps\pol$ are $L_k^{\ps}= \langle x,\nu_k\rangle + \ps c_k $. Consider the following family of functions on $P$
$$u^{\ps}= u_o -\frac{u_o^\ps}{\ps}.$$ 
We will show that for $s>1$, $u^{\ps}\in \mS(\pol,\nu)$. Since $u_o^\ps$ is smooth on $\pol$ when $\ps>1$, to show that $u^{\ps}\in \mS(\pol,\nu)$ we need to show that $G^\ps=\mbox{Hess }u^\ps$ is positive definite on $\mathring{\pol}$. This is clear since $L_k^{\ps}(x)> L_k(x)$ on $\pol$ and 
$$G^\ps = \frac{1}{2}\sum_{k=1}^d \left(\frac{1}{L_k} - \frac{1}{\ps L_k^\ps}\right)\nu_k\otimes \nu_k.$$
Given a face of the polytope $F$, a similar argument using only the $L_i$'s which do not vanish identically over $F$, would show that the restriction of ${Hess }(u^\ps)$ to $F$ is positive definite on $F$.~\\

In~\cite{abreufreitas}, the authors show that for the $2$--sphere, $\lambda_1^T(g_{u^{\ps}})\nearrow +\infty$ when $\ps$ goes to $1$. We will use another approach to show that the same holds in higher dimension. The rough idea is that, since $u^\ps \ra 0$ uniformly on $P$, the eigenvalues of the inverse of its Hessian should, in some way, tend to infinity and thus the Rayleigh quotient of any function should go to infinity. We write 
\begin{itemize}
\item $G_o^\ps= \Hess u_o^\ps$, and $H_o^\ps=(G_o^\ps)^{-1}$\\
\item $G^\ps=\Hess u^\ps$ and $H^\ps=(G^\ps)^{-1}$. 
\end{itemize}

We start by proving the following simple lemma.
\begin{lemma}
For any $f\in C^1(\pol)$ 
\begin{equation}\label{1stOrderBound}\int_\pol H^\ps(df,df) \vol \geq  \int_\pol H_o(df,df)\vol.\end{equation}
In particular, the variation $\lambda_1^T(\ps ):= \lambda_1^T(g_{u^{\ps}})$ is bounded below by $\lambda_1(g_{u_o})$
\end{lemma}
\begin{proof}
First note that {since} $H_o$, $H_o^\ps$, $G_o^\ps$, $G^\ps$ are symmetric matrices they have real eigenvalues. Moreover, the eigenvalues of $H_o G_o^\ps$ are, strictly smaller than $1$ on $\mathring{\pol}$. Indeed, if $\lambda$ is an eigenvalue for $H_o G_o^\ps$ and $u$ is the corresponding eigenvector, then $G_o^\ps u=\lambda G_o u$, so that 
$$\sum_{k=1}^d \left( \frac{1}{L^{\ps}_k(x)} -\frac{\lambda}{L_k(x)} \right) \langle \nu_k,u\rangle^2=0,$$
which is not possible if $\lambda\geq 1$ since $L^{\ps}_k(x)>L_k(x)$ on $\pol$. Because $H_o G_o^\ps$ is symmetric as each $H_o$ and $G^\ps$ are, this implies that $||H_o G_o^\ps||<1$. Since $G^\ps =G_o-\frac{1}{s}G_o^\ps$ and $s>1$, we have
$$H^\ps= \left(Id_n + \sum_{k=1}^\infty \left(\frac{1}{\ps} H_o G_o^\ps\right)^k\right)H_o,$$ on the interior of $\pol$. From this expression, we get that the inequality \eqref{1stOrderBound} holds for any $f\in C^1(\pol)$ and $s>1$. \end{proof}
We are now in a position to prove that for the family of metrics determined by $u^\ps$, $\lambda_1^T$ is unbounded.
\begin{proposition} $\sup\{ \lambda_1^T(\ps) \,|\,\ps>1\} = +\infty$.
\end{proposition}

\begin{proof} 
Assume that $\lambda_1^T(\ps)$ is also bounded above by a constant, say $\kappa>0$, then, we can find a sequence $\ps_k \ra 1^+$ such that $\lambda_1^T(\ps_k)$ converges to some $\lambda >0$.

Consider a sequence $f_{\ps_k}\in C^{\infty}(M)^T=C^{\infty}(\pol)$ of eigenfunctions $$\Lap^{g_{u^{\ps_k}}}f_{\ps_k}= \lambda_1^T(\ps_k)f_{\ps_k}$$ normalized such that $\|f_{\ps_k}\|_{L^2}= \int_{\pol} (f_{\ps_k})^2 \vol = 1$. Note that the inequality~\eqref{1stOrderBound} implies that the Sobolev  norms of $\{f_{\ps_k}\}$ in $H^{1}(M, g_{u_o})$ are bounded above by $\kappa +1$. Indeed, combining the hypothesis and \eqref{1stOrderBound}, we have
$$\kappa> \lambda_1^T(\ps_k)= \int_\pol H^\ps(df_{\ps_k},df_{\ps_k}) \vol \geq \|\nabla_{g_{u_o}}f_{\ps_k}\|_{g_{u_o}}^2.$$

Consequently, there exists a subsequence, that we still index by $\ps$ for simplicity, of eigenfunctions $f_{\ps}\in C^{\infty}(M)^T$ converging in the $L^2(M,g_{u_o})$ topology to some function $f\in L^2(M)$. We have $\|f\|_{L^2}=1$, $\int_\pol f(x) \vol =0$. 

~\\

A straightforward calculation yields
$$G^\ps(x) = \frac{\ps-1}{2\ps} \sum_{k=1}^d \left(\frac{L_k(x) + \ps c_k}{L_k(x)L_k^\ps(x)} \right)\nu_k\otimes \nu_k,$$ and thus, for any $x\in\mathring{\pol}$, 
$$B_x := \, \lim_{\ps \ra 1^+}\frac{G^\ps(x)}{\ps-1} \, = \frac{1}{2} \sum_{k=1}^d \left(\frac{L_k(x) + c_k}{L_k(x)^2} \right)\nu_k\otimes \nu_k$$ is positive definite and depends smoothly on $x\in\mathring{\pol}$.  For $x\in\mathring{\pol}$, let $$A_x= \, \lim_{\ps \ra 1^+} (\ps -1)H^\ps(x)$$ be the inverse of $B_x$. Notice that $(\ps -1)H^\ps(x)$ converges to its limit uniformly on compact subsets in $\pol$.

Let $K$ be a compact subset of $\mathring{\pol}$. The integral $$\int_K H^\ps(df_{\ps},df_{\ps})\vol$$ can be written as
 \begin{equation}\label{boundonK} 
 \int_K A_x\left(\frac{df_{\ps}}{\sqrt{\ps-1}},\frac{df_{\ps}}{\sqrt{\ps-1}}\right) \vol+\int_K \left((s-1)H^\ps-A_x\right)\left(\frac{df_{\ps}}{\sqrt{\ps-1}},\frac{df_{\ps}}{\sqrt{\ps-1}}\right) \vol.
 \end{equation}
Now for any $\epsilon>0$ 
\begin{IEEEeqnarray*}{c}
\left| \int_K \left((s-1)H^\ps-A_x\right)\left(\frac{df_{\ps}}{\sqrt{\ps-1}},\frac{df_{\ps}}{\sqrt{\ps-1}}\right) \vol\right|\\
\leq \mbox{sup}_K ||(s-1)H^\ps-A_x|| \int_K \left|\frac{df_{\ps}}{\sqrt{\ps-1}}\right|^2 \vol\\
\leq \epsilon \int_K \left|\frac{df_{\ps}}{\sqrt{\ps-1}}\right|^2\vol
\end{IEEEeqnarray*}
when $s$ is sufficiently close to $1$. On $K$, the symmetric bilinear form $A$ is positive definite and its smallest eigenvalue is strictly positive. Hence, on $K$, the norm of $A_x$ is equivalent to the Euclidean norm i.e.    
$$\int_K A\left(\frac{df_{\ps}}{\sqrt{\ps-1}},\frac{df_{\ps}}{\sqrt{\ps-1}}\right)\vol\geq \Gamma_K \int_K \left|\frac{df_{\ps}}{\sqrt{\ps-1}}\right|^2\vol$$ for some constant $\Gamma_K$. 
The inequality 
$$\int_K H^\ps(df_{\ps},df_{\ps})\vol \leq \kappa,$$
implies 
$$
 \int_K A_x\left(\frac{df_{\ps}}{\sqrt{\ps-1}},\frac{df_{\ps}}{\sqrt{\ps-1}}\right) \vol+\int_K \left((s-1)H^\ps-A_x\right)\left(\frac{df_{\ps}}{\sqrt{\ps-1}},\frac{df_{\ps}}{\sqrt{\ps-1}}\right) \vol\leq \kappa,
 $$
 but
 \begin{IEEEeqnarray*}{c}
 \int_K A_x\left(\frac{df_{\ps}}{\sqrt{\ps-1}},\frac{df_{\ps}}{\sqrt{\ps-1}}\right) \vol+\int_K \left((s-1)H^\ps-A_x\right)\left(\frac{df_{\ps}}{\sqrt{\ps-1}},\frac{df_{\ps}}{\sqrt{\ps-1}}\right) \vol\\
 \geq (\Gamma_K-\epsilon) \int_K \left|\frac{df_{\ps}}{\sqrt{\ps-1}}\right|^2\vol
 \end{IEEEeqnarray*}
and we conclude that 
$$
\frac{\Gamma_K}{2} \int_K \left|\frac{df_{\ps}}{\sqrt{\ps-1}}\right|^2\vol\leq \kappa.
$$

Using the Poincar\'e inequality, there exists $C_K$, a constant depending only on $K$, such that $$C_K\int_K \left|\frac{df_{\ps}}{\sqrt{\ps-1}}\right|^2\vol \geq \frac{1}{\ps-1}\int_K \left(f_{\ps} - \frac{\int_Kf_\ps}{\int_K\vol}\right)^2\vol.$$ However, since $f_\ps \ra f$ in the $L^2$--topology on $\pol$  and on $K$, $$\int_K \left(f_{\ps}  - \frac{\int_Kf_\ps\vol}{\int_K\vol}\right)^2\vol\; \longrightarrow \; \int_K \left(f  - \frac{\int_Kf\vol}{\int_K\vol}\right)^2\vol.$$ 
But
$$ 
0\leq \int_K \left(f_{\ps}  - \frac{\int_Kf_\ps\vol}{\int_K\vol}\right)^2\vol \leq \frac{2(s-1)C_K\kappa}{\Gamma_K} \longrightarrow 0,
$$
{when $s\ra 1^+$}. This implies that 
$$
\int_K \left(f  - \frac{\int_Kf\vol}{\int_K\vol}\right)^2\vol=0,
$$ 
and $f$ is a constant on $K$. Since $K$ is arbitrary, $f$ is constant on $\pol$. But $\int_P f=0$ so that $f$ must be identically zero which contradicts $\int_P f^2=1$. 
\end{proof}
This proposition proves the second part of Theorem \ref{THEOminLAMBDA1}. We have thus proved Theorem \ref{THEOminLAMBDA1}.

\section{Bounds on $\lambda_1$ for toric manifolds}\label{bound_lambda_1}

\subsection{The Bourguignon--Li--Yau bound of an integral polytope}\label{BLYboundsub}
Consider a complex projective manifold $(M,J, L)$ where $(M,J)$ is complex manifold and $L\rightarrow M$ is a very ample line bundle giving a fixed embedding $\Phi: M\hookrightarrow \bC\bP^N\simeq \bP(H^0(M,L))^\ast$. In \cite{BLY}, Bourguignon--Li--Yau gave a bound on the first eigenvalue of any K\"ahler metric $\omega$, compatible with $J$. The bound depends only on the dimension of $M$, the K\"ahler class $[\omega]\in H^{2}(M,\bR)$ and the embedding class $[\Phi^*\omega_{FS}] = 2\pi c_1(L)$. The aim of this subsection is to discuss and review the result, as well as apply it to integral toric manifolds. We start by stating the main result of \cite{BLY}.
\begin{theorem}[Bourguignon--Li--Yau]\label{bly}
Let $M^n$ be a compact complex manifold and let $\Phi:M\rightarrow \bC\bP^N$ be a holomorphic immersion such that $\Phi(M)$ is not contained in any hyperplane in $\bC\bP^N$. Then for any K\"ahler metric $\omega$ on $M$, compatible with the given complex structure 
\begin{equation}\label{BOUNDbly}
 \lambda_1(M,\omega)\leq \frac{2n(N+1)\int_M\Phi^*\omega_{FS}\wedge \omega^{n-1}}{N\int_M \omega^n},
\end{equation} where $\omega_{FS} = {\bi}\partial \bar{\partial} \log (|Z_0|^2 + \dots + |Z_N|^2)$ is the Fubini-Study form on $\bC\bP^N$.
\end{theorem}

We say that an immersion $\Phi:M\rightarrow \bC\bP^N$ is {\it full}  if its image is not contained in any hyperplane of $\bC\bP^N$. Given a full holomorphic immersion, we set $$\bfB([\Phi],[\omega])=\frac{2n(N+1)\int_M\Phi^*\omega_{
FS}\wedge \omega^{n-1}}{N\int_M \omega^n}.$$  It is clear that $\bfB([\Phi],[\omega])$ only depends on the $H^{1,1}(M,\bR)$--cohomology classes $[\omega]$ and $[\Phi^*\omega_{FS}]$.\\

Arezzo--Ghigi--Loi generalized Theorem~\ref{bly} to provide a bound on the first eigenvalue of K\"ahler manifolds admitting a Gieseker stable bundle, see~\cite{AGL}. In \cite{leonidP}, Polterovich used the Bourguignon--Li--Yau Theorem to give a bound on $\lambda_1$ for all symplectic manifolds whose symplectic class $\frac{1}{2\pi}[\omega]$ lies in $H^2(M,\bQ)$. In the toric case, this theorem can be applied directly to provide (various) bounds on the first eigenvalue of compact toric K\"ahler manifolds. Indeed, given a toric compact K\"ahler manifold $(M,\omega, J,T)$ it is known, see for e.g. \cite{CLS:toricVarietiesBOOK,guillMET}, that $H^2_{dR}(M) = H^{1,1}_{\bar{\del}}(M)$. Hence, one can pick a symplectic form $\tilde{\omega}$, compatible with $J$ and lying in an integral and very ample class. Using Kodaira's embedding Theorem, we know that there exists a full embedding $\Phi: M\ra \bcp^N$ such that $[\Phi^*\omega_{FS}]=[\tilde{\omega}]$. Hence, by Theorem~\ref{bly} $\lambda_1(\omega, J)\leq \bfB(\Phi,[
\omega])$. The input of Theorem \ref{thm_bound_lambda_1} is to give a bound that depends only (and explicitly) on the polytope.

\begin{rem} Using Bourguignon--Li--Yau Theorem, we can get a finer  upper bound for $\lambda_1(M,\omega)$, namely 
$$ \inf\left\{\bfB([\Phi],[\omega])\, |\, \; \Phi : M\rightarrow \bC\bP^N \;\; \mbox{full holomorphic immersion}\right\}.$$ 
Many natural questions arise: given $\Omega =[\omega]$, is this infimum reached for some immersion ? If so, is this immersion minimal or balanced ?  Note that  in the Riemannian case, there is an Embedding Theorem due to Colin-de-Verdi\`ere and El Soufi--Ilias (see \cite{EI})  concerning $\lambda_1$--extremal metrics. These Riemannian metrics are essentially defined as critical points of the map $g\mapsto \lambda_1(g)$ on the space of Riemannian metrics with fixed total volume. In that case, the aforementioned authors showed that an orthonormal basis of the first eigenspace provides a {\it minimal} embedding into a sphere $S^N$ such that the standard round metric on $S^N$ pulls-back to the extremal one.
\end{rem}

 \begin{definition}
Given a compact symplectic toric manifold $(M,\omega, T)$ with moment polytope $\pol$ we say that $\pol \subset \kt^*$ is {\it integral} if  its vertices lie in the dual of the lattice $\Lambda\subset  \kt$ of circle subgroups of the torus $T$.
\end{definition}

It is well known that integral polytopes correspond to symplectic toric manifolds whose cohomology class is integral. Moreover, any such toric manifold admits a compatible toric K\"ahler structure \cite{delzant:corres,guillMET} and an explicit equivariant holomorphic full embedding into some projective space, see for e.g. \cite[Theorem 6.1.5]{CLS:toricVarietiesBOOK}. We will apply Bourguignon--Li--Yau Theorem using this embedding. {For the sake of completeness} we recall {some facts about} this embedding.

Any symplectic toric manifold admits compatible complex structures (see \cite{delzant:corres,guillMET}). {An example of such a complex structure is given by the Guillemin metric} $g=g_{u_o}$ corresponding to the Guillemin potential $u_o$ {(see Definition~\ref{guilleminPOTENTIAL})}. Moreover, any two such compatible complex structures are biholomorphic (see Remark~\ref{biholomBETWEENstruct} below). 

We start with a smooth K\"ahler toric manifold $(M,g,\omega, J)$ whose cohomology class $[\omega]$ is integral and {such} that $g=g_u$ for some symplectic potential $u\in \mS(P,\nu)$. On the underlying toric variety $(M,J)$ the class $[\omega]$ corresponds to an ample divisor which is then very ample  \cite[Theorem 6.1.5]{CLS:toricVarietiesBOOK}.    
{More} precisely, $(M,J)$ carries a holomorphic line bundle $L$ whose first Chern class is $[\omega]$ and {which} defines a full holomorphic embedding $\Phi_{u} : M\hookrightarrow \bcp^N$ where $N+1$ is the number of lattice points in $\pol$. The embedding is associated to a basis of $H^0(L)$ namely $\{e^{mz}\mathbb{1}, m\in \Lambda^*\cap \pol\}$ where $\mathbb{1}$ is a reference holomorphic section of $L$ and may be defined by  
 \begin{equation} \label{toricEMBEDDING_holo}                                                                                                                                                                                                                                                                    
\Phi_u(z) =  [e^{\w_0 \cdot z}:\cdots :e^{\w_N \cdot z}],                                                                                                                                                                                                                                                                             
\end{equation}
where $\w_0,\cdots,\w_N$ are the lattice points in $\pol$ and $z=y+\bi\theta$ are local holomorphic coordinates on $\mathring{M}$, see Remark~\ref{remCPLXcoor}.  

One can express such an embedding in action-angle coordinates via $z=y+\bi\theta = {\frac{\partial u}{\partial x}}+\bi \theta$ and for convenience we denote each coordinate 
 \begin{equation} \label{toricEMBEDDING0}                                                                                                                                                                                                                                                                         
{ \Phi_u(x,\theta) =  [e^{\w_0\cdot {\frac{\partial u}{\partial x}}}e^{\w_0\cdot \bi \theta}: \dots :e^{\w_N\cdot {\frac{\partial u}{\partial x}}}e^{\w_N\cdot \bi \theta}], }                                                                                                                                                                                                                                                                           \end{equation} because  $e^{\w_0\cdot z}=  e^{\w_0\cdot {\frac{\partial u}{\partial x}}}e^{\w_0\cdot \bi \theta}$.

\begin{rem}\label{biholomBETWEENstruct} It is known, see for example~\cite{don:Large}, that for two distinct symplectic potentials $u, u_o\in \mS(P,\nu)$ the map $$\gamma_{u,u_o} : P\times T\longrightarrow P\times T $$ defined by $ \gamma_{u,u_o}(x,e^{\bi\theta }) = (\left({\frac{\partial u_o}{\partial x}}\right)^{-1} {\frac{\partial u}{\partial x}}, e^{\bi\theta })$ extends as an equivariant diffeomorphism on $M$ sending $J_u$ to $J_{u_o}$ and $\gamma_{u,u_o}^*\omega = \omega + dd^ch$ where $h\in C^{\infty}(M)^T$.  
\end{rem} 
 \begin{rem}\label{rem_INDEPclass} It has been proved by Guillemin in~\cite{guillMET} that $[\Phi_{u_o}^*\omega_{FS}] = [\omega]$. Moreover, considering the diffeomorphism $\gamma_{u,u_o}\in \mbox{Diffeo}(M)$ of Remark~\ref{biholomBETWEENstruct}, we have $$\gamma_{u,u_o}^* \Phi_{u_o} = \Phi_{u}$$ {Since the space of symplectic potentials is convex (i.e $u_t = tu +(1-t)u_o \in \mS(P,\nu)$ for $t\in [0,1]$) the map $\gamma_{u,u_o}$ lies in the connected component of the identity in $\mbox{Diffeo}(M)$}. In particular, it preserves cohomology classes and $[\Phi_{u}^*\omega_{FS}] = [\omega]$. \end{rem}

\begin{rem}\label{embTORUS}
Together with $\Phi_u$ comes an embedding $\phi : T\hookrightarrow \bT^{N+1}$ induced by the linear map $\phi_* : \kt \ra \bR^{N+1}$, taking $\theta\in \kt$ to $$(\theta\cdot \w_0,\dots, \theta\cdot \w_N) \in \bR^{N+1}$$ which is clearly injective. The maps $\Phi_u$ are $\phi$--equivariant embeddings. % since $\Span_{\bZ} \{ \w \}_{\w\in P\cap \Lambda^*} =\Lambda^*$
\end{rem}

{It follows from Remark~\ref{rem_INDEPclass}, that the class of $[\Phi_{u}^*\omega_{FS}] $ does not depend on the chosen symplectic potential $u\in\mS(P,\nu)$. It also follows that for any $u\in\mS(P,\nu)$, $\Phi_u$ is a full {holomorphic embedding} iff $\Phi_{u_o}$ is.}

{Let us consider the Guillemin potential $u_o$. 
\begin{proposition}
Given a symplectic toric manifold, the map  $\Phi_{u_0}$ above is a well defined full holomorphic embedding.
\end{proposition}}
{This is a well known fact. We write the details down here for the reader's convenience. See \cite{guillMET} and {\cite{CLS:toricVarietiesBOOK}} for more on this.}
\begin{proof}
Since $u_o= \frac{1}{2}\sum_{k=1}^d (L_k \log L_k -L_k)$, then substituting $$y= {\frac{\partial u_o}{\partial x}} = \frac{1}{2} \sum_{k=1}^d(\log L_k(x)) \nu_k, $$ we get that $$\Phi_{u_0}(x,\theta) = [\Pi_{i=1}^d L_k(x)^{\frac{\w_0\cdot \nu_k}{2}}e^{\bi\w_0\cdot \theta } :\cdots : \Pi_{i=1}^d L_k(x)^{\frac{\w_N\cdot \nu_k}{2}}e^{\bi\w_N\cdot \theta }].$$ 
{The homogeneous coordinates in} this last expression are not smooth on $M$ because $y={\frac{\partial u}{\partial x}}$ blows up when $x$ approaches the boundary and $\theta$ is {only} well-defined modulo a lattice. To see it is globally defined on $M$, we multiply each {homogeneous} coordinate of $\Phi_{u_o}$ by the function $\Pi_{i=1}^d L_i(x)^{c_i/2}$ where $L_i(x)= x\cdot\nu_i +c_i$ and we get  
\begin{equation}\label{toricEMBEDDING_aa}
 \Phi_{u_o}(x,\theta)=\left[\left(\Pi_{i=1}^d (L_i(x))^{\frac{L_i(m)}{2}}\right)e^{\bi m\cdot \theta}\right]_{m\in \pol\cap \bZ^n},
\end{equation} where we identify $\kt\simeq \bR^n$ and $\Lambda^* \simeq \bZ^n$.\\

$\bullet$ The homogeneous coordinates of $\Phi_{u_o}$ do not vanish simultaneously. It is clear that the homogeneous coordinates of $\Phi_{u_o}$ do not vanish over the interior of $\pol$. Let $x\in \partial \pol$.  Without loss of generality assume that $L_1(x)=\cdots=L_r(x)=0$ and suppose that $x$ is in the interior of the face $F$ defined by $L_1(x)=\cdots=L_r(x)=0$. Because $P$ is integral there is also a point $m \in F \cap \bZ^n$ so that $L_1(m)=\cdots=L_r(m)=0$ and the $m$-th coordinate does not vanish. \\

$\bullet$ $\Phi_{u_o}$ is globally defined on $M$. The application $\theta : \mathring{M} \ra \bR^n/\bZ^n$ is well defined on the interior of $P$ but not on the boundary. Let $x\in \partial \pol$ say $L_1(x)=\cdots=L_r(x)=0$. Then, we should check that
$$
\Phi_{u_o}(x,\theta)=\Phi_{u_o}\left(x,\theta+ \sum_{l=1}^r\alpha_l\nu_l\right),
$$ 
because the Lie algebra of the subgroup of $\bT^n$ that fixes the points in $x^{-1}(F)$ is spanned by $\nu_1,\cdots \nu_r$. The only homogeneous coordinates that do not vanish at $x$ are those corresponding to $m$'s such that $L_i(m)=0$ for $i=1,\cdots, r$. Let $m_a$ be such a point. Then $\sum_{l=1}^r\alpha_l\nu_l = \sum_{l=1}^r\alpha_l (-c_l)$ because $m_a\cdot \nu_l=-c_l$ is independent of $m_a$, that is each component of $\Phi_{u_o}(x,\theta+ \sum_{l=1}^r\alpha_l\nu_l)$ is $\left(e^{\frac{\bi}{2}\sum_{l=1}^r\alpha_l (-c_l)}\right)$ times the corresponding component of $\Phi_{u_o}(x,\theta)$ and thus $\Phi_{u_o}(x,\theta+ \sum_{l=1}^r\alpha_l\nu_l)=\Phi_{u_o}(x,\theta)$.\\

$\bullet$ $\Phi_{u_o}$ is holomorphic because on the interior of $\pol$ it coincides with
 $$
 [e^{m\cdot z}]_{m\in \pol\cap \bZ^n},
 $$
 which is expressed in terms of complex coordinates as a holomorphic function. \\

$\bullet$  $\Phi_{u_o}$ {is an injective immersion}. It is clear over the interior of $P$ thanks to Remark \ref{embTORUS}. Over the boundary, we may assume that $\pol$ is standard around one of its vertices so that $L_1=x_1-a_1,\cdots, L_n=x_n-a_n$. In this case by reordering if necessary we may assume that $m_1-m_0=(1,0,\cdots,0), \cdots, m_n-m_0=(0,\cdots,0,1)$. To prove injectivity of $\Phi_{u_o}$ and its derivative it is enough to prove injectivity of 
$$
z\rightarrow (e^{(m_1-m_0)z},\cdots, e^{(m_n-m_0)z}).
$$
But the above is simply 
$$
z\rightarrow (e^{z_1},\cdots, e^{z_n})
$$ which is injective modulo $2\pi \bi \bZ^n$ as expected and has injective derivative. 
\end{proof}

 Applying the Bourguignon--Li--Yau theorem to $(M^{2n},\omega, g_u,J, T)$, we get that 
  \begin{equation}\label{BLYboundTORIC}
   \lambda_1(\omega,g_u) \leq \frac{{2}n(N+1)}{N} \int_M \frac{\Phi_{u}^*\omega_{FS}\wedge \omega^{n-1}}{\omega^n} = \frac{2n(N+1)}{N}
  \end{equation} where $N+1$ is the number of lattice points in $\pol$. 

  \begin{rem}\label{rem:optimalCALSS}
Observe that taking $k\pol$ for $k\in \bN^*$ and $k\geq 2$ implies that the right hand side of~\eqref{BLYboundTORIC} decreases to $2n$. However, the left hand side decreases quickly as well since $\lambda_1(k\omega,kg_u) = \frac{1}{k}\lambda_1(\omega,g_u)$. Hence, in each rays of K\"ahler cone in $H^{1,1}(M,\bZ)$ there is an optimal class, the primitive class, on which we may apply the bound $\bfB(\Phi_u,[\omega_\pol])$. 
 \end{rem}
\begin{rem}\label{integralBOUND} The Bourguignon--Li--Yau bound is an integer if and only if $N=2$, $N=n$ or $N=2{n}$. The two first cases imply $M$ is a projective space and the last one gives $\lambda_1(\omega,g_u) \leq 2n+1$. Note that, in this last case, the first eigenvalue of a K\"ahler--Einstein metric, which is $2\lambda$ by Proposition~\ref{propKEeigenfunction} where $\lambda = 2\pi c_1(M)/[\omega]$, cannot reach this bound whenever $[\omega]$ is integral. \end{rem}

\subsection{A bound on $\lambda_1$ for toric manifolds}\label{sectBoundtoric}
Let $(M^{2n},\omega, g,J)$ be a compact toric K\"ahler manifold. The cohomology class of $\frac{\omega}{2\pi}$ is integral if and only if $\pol$ is integral. In this subsection we will not assume that $\pol$ is integral. We start by defining an integer $k_0(\pol)$ associated to $\pol$. Let $k$ be a fixed integer. Consider the lattice $\bZ^n/k\cap \pol$.

\begin{definition}
Let $\pol$ be a Delzant polytope. Set
$$
\Lm_{i,k}=\text{min}\{L_i(m), \, m\in \pol \cap \bZ^n/k \}.
$$
\end{definition}
 Let $\pol_k$ be the polytope defined by the inequalities $L_i(x)\geq \Lm_{i,k}$ i.e. 
$$
\pol_k=\{x\in \pol: L_i\geq \Lm_{i,k}, \, i=1, \cdots d \}.
$$
Note that if $\pol_k$ is a non-empty polytope with $d$ facets then these facets are parallel to those of $P$.  

We want to show that as $k$ tends to infinity the lattice $\bZ^n/k\cap \pol$ becomes finer and eventually $\pol_k$ will look combinatorially like $\pol$.
\begin{lemma}\label{p_k_and_p}
Let $\pol$ be a simple compact Delzant polytope 
$$
\pol=\{x\in \bR^n: L_i\geq 0, \, i=1, \cdots d \}.
$$
Let $\Lm_{i,k}$ and $\pol_k$ be defined as above then $\Lm_{i,k}\rightarrow 0$ as $k$ tends to infinity and $\pol_k$ has the same combinatorial type as $\pol$ for $k$ large enough.
\end{lemma}

\begin{proof} 
{Assume without loss of generality that $i=1$ i.e. we want to prove that $\Lm_{1,k}\rightarrow 0$ as $k$ tends to infinity. Choose a vertex of $\pol$ say $a=(a_1,\cdots,a_n)$ on the first facet of $\pol$ that is $L_1(a)=0$. Assume furthermore that the first $n$ facets of $\pol$ meet at $a$. This is possible as we can simply relabel the facets.}. Since $\pol$ is Delzant then there is $A\in SL(n,\bZ)$ such that $A\pol$ is standard at $a$. That is 
$$
A\pol=\{x\in \bR^n: x_1-\tilde{a}_1\geq 0, \cdots,  x_n-\tilde{a}_n\geq 0, \tilde{L}_{n+1}(x)\geq 0,\cdots \tilde{L}_d(x)\geq 0\},
$$
where $\tilde{L}_{n+1},\cdots, \tilde{L}_d$ are affine functions and $\tilde{a}=(\tilde{a}_1, \cdots, \tilde{a}_n)=Aa$. Note that $\tilde{L_i}(\tilde{a})>0$ for all $i=n+1,\cdots, d$ because {$\{\tilde{a}\}$} is the intersection of the first $n$ facets of $A\pol$. Pick $a_k=(a_{1,k},\cdots,a_{n,k})\in \bZ^n$ in the following way
$$
\frac{a_{i,k}-1}{k}\leq \tilde{a}_i\leq \frac{a_{i,k}}{k},\, \forall i=1,\cdots, n.
$$
Now $a_k/k$ satisfies the first $n$ inequalities in the definition of $A\pol$ by construction. On the other hand $|a_k/k-\tilde{a}|\leq n/k$ so that when $k$ is sufficiently large {and $i>n$}, $\tilde{L}_i(a_k)$ is close to $\tilde{L_i}(\tilde{a})>0$ hence $\tilde{L}_i(a_k)>0$ for large enough $k$. This implies $a_k/k\in A\pol$ therefore $a_k/k\in \bZ^n/k\cap A\pol$. But $A^{-1}a_k/k\in \bZ^n/k\cap \pol$ where we have used the fact that $A\in SL(n,\bZ)$. Also $|{A^{-1}a_k/k}-a|\leq||A^{-1}||n/k$ and therefore tends to zero as $k$ tends to infinity. It follows that {$L_1({\frac{A^{-1}a_k}{k}})\rightarrow L_1(a)=0$ because $L_1$} is a continuous function. Since {$\Lm_{1,k}\leq L_1({\frac{A^{-1}a_k}{k}})$} and is positive, we prove our claim.

\end{proof}

\begin{definition}
We define $k_0(\pol)$ to be smallest integer $k\geq 1$ such that $\pol_k$ has the same combinatorial type as $\pol$. 
\end{definition}
For any integer $k$, we set $N_k=\sharp(\bZ^n/k\cap \pol)-1$. Note that if $\pol$ is integral, $k_0(\pol)=1$. {Recall that we write $L_i(x)=x\cdot\nu_i+c_i$.}

\begin{lemma} Let $\pol$ be a Delzant polytope and $k{\geq} k_0(\pol)$ then $k\pol_k$ is an integral Delzant polytope such that $N_k +1 = \sharp((k\pol_k) \cap \bZ^n)$. 
\end{lemma}

\begin{proof} For $i\in\{1,\dots, d\}$, we denote $\tilde{c}_{i,k}= c_i-\Lm_{i,k}$ so that $\pol_k= \{x\in\bR^n\,|\, x\cdot\nu_i + \tilde{c}_{i,k} \geq 0, \, \mbox{ for } i=1,\dots, d\}$. Then \begin{equation}\begin{split} k\pol_k &= \{x\in \bR^n\,|\, (x/k)\cdot\nu_i + \tilde{c}_{i,k} \geq 0, \, \mbox{ for } i=1,\dots, d\}\\
&= \{x\in \bR^n\,|\, x\cdot\nu_i + k\tilde{c}_{i,k} \geq 0, \, \mbox{ for } i=1,\dots, d\}.
\end{split}
\end{equation}
 Therefore, $k\pol_k$ has the same normal inward vectors as $\pol$ and is Delzant {iff} $\pol$ is. We are using the fact that that $k\geq k_0(\pol)$ so that $k\pol_k, \pol_k$ and $\pol$ have the same combinatorial type. Moreover $k\tilde{c}_{i,k}\in\bZ$. Indeed, the compacity of $\pol$ implies that there exists $m/k\in P$ such that $m\in\bZ^n$ and $\Lm_{i,k}= L_i(m/k)$ and thus $$k\tilde{c}_{i,k} = k(c_i-\Lm_{i,k})= k(c_i-L_i(m/k))= k(c_i-((m/k)\cdot\nu_i +c_i)) = -m\cdot\nu_i \in\bZ.$$ Hence, $k\pol_k$ is an integral Delzant polytope. Finally, $m \in ((k\pol_k) \cap \bZ^n)$ if and only if  ${\frac{m}{k}} \in ((\pol_k) \cap (\bZ^n/k))\subset (\pol \cap (\bZ^n/k))$, thus $\sharp((k\pol_k) \cap \bZ^n)\leq \sharp(\pol\cap (\bZ^n/k))$. Now any point in $(\pol \cap (\bZ^n/k))\backslash (\pol_k \cap (\bZ^n/k))$ would contradict the minimality of $\Lm_{i,k}$ for some $i\in\{1,\dots, d\}$. Hence $\sharp((k\pol_k) \cap \bZ^n)= \sharp ((\pol_k) \cap (\bZ^n/k)) =\sharp(\pol\cap (\bZ^n/k))$. \end{proof}

{We are now in a position to prove Theorem \ref{thm_bound_lambda_1}}
\begin{proof} (of Theorem \ref{thm_bound_lambda_1}) At this point we can apply Bourguignon--Li--Yau's theorem \ref{bly} to conclude that
\begin{equation}\label{bound_bly}
{\lambda_1(M,\omega)\leq \frac{2n(N_k+1)\int_\pol\Phi_{u,k}^*\omega_{FS}\wedge \omega^{n-1}}{N_k\int_\pol \omega^n},}
\end{equation} where $\Phi_{u,k}$ is the embedding~\eqref{toricEMBEDDING0} associated to $k\pol_k$. Now it is well known that the symplectic class associated to $k\pol_k$ is $k$ times the one associated to $P_k$. Then, see Remark \ref{rem_INDEPclass}, we have $[\Phi_{u,k}^*\omega_{FS}] = k[\omega_k]$ where $\omega_k$ is the symplectic form determined by $P_k$. Recall that $\pol_k=\{x\in P: L_i\geq {\Lm_{i,k}}, \, i=1, \cdots d \}$ has the same combinatorial type as $P$, by assumption on $k$. In particular, $P$ and $P_k$ share the same normal inward vectors ({hence the same fan}) but the constant parts of the affine-linear functions defining $P_k$ are $c_i-{\Lm_{i,k}}$ for $i=1,..., d$.  

Let $E_i$ be the cohomology class that is Poincar\'e dual to the divisor in $M$ whose image under the moment map is the facet $F_i$. Then
$$
\frac{1}{2\pi \bi}[\partial\bar\partial \log L_i]=E_i  \;\; \;
$$ 
(see \cite{guillMET}, Theorem 6.2 for a statement of this elementary fact). % and we conclude that $[\Phi_{u_o,k}^*\omega_{FS}]=\newp{-2\pi} k\sum_{i=1}^d(-c_i+{\Lm_{i,k}})E_i$. 

It is well know (see \cite{guillMET}) that $\frac{[\omega]}{2\pi}=\sum_{i=1}^d c_iE_i$. Applying this to $P_k$ we get 
\begin{equation} \frac{[\Phi_{u,k}^*\omega_{FS}]}{2 \pi}= k(c_i -{\Lm_{i,k}})E_i=k\left([\omega]-\sum_{i=1}^d{\Lm_{i,k}}E_i\right).
\end{equation}
Replacing in equation (\ref{bound_bly}) we see that
$$
\lambda_1(M,\omega)\leq \frac{2n(N_k+1)}{N_k}\left(k-\frac{k\int_\pol\sum_{i=1}^d {\Lm_{i,k}} E_i\wedge \omega^{n-1}}{\int_\pol \omega^n}\right).
$$ 

Because $E_i$ is Poincar\'e dual to the pre-image of the $i$-th facet, which we denote by $D_i$, we have that 
$$
\int_\pol E_i\wedge \omega^{n-1}=\int_{D_i}\omega^{n-1}=\text{vol}_{n-1}(D_i)>0.
$$
It follows that 
$$
\lambda_1(M,\omega)\leq\frac{2nk(N_k+1)}{N_k}. 
$$
It is clear from the above argument that if $\pol$ is integral one can take $k=1$ and the bound becomes
$$
\lambda_1(M,\omega)\leq \frac{2n(N+1)}{N}.
$$ This proves Theorem~\ref{thm_bound_lambda_1}.
\end{proof}

 \subsection{The equality case in Bourguignon--Li--Yau's bound}
 The goal of this subsection is to study K\"ahler toric metrics which saturate the bound in Theorem \ref{thm_bound_lambda_1}. We give a quick overview of Bourguignon, Li and Yau's proof in~\cite{BLY} as this proof will be important to us.  
 \begin{proof} [Sketch of proof of Theorem \ref{bly}, see \cite{BLY}]
 Recall that the first eigenspace of $(\bcp^N,\omega_{FS})$ has a basis given by the functions $[Z] \mapsto \Psi_{ij}(Z) -\frac{\delta_{ij}}{N+1}$ where for $i,j \in \{0,1,\dots, N\}$, $$\Psi_{ij}(Z)=\frac{Z_i\overline{Z}_j}{\sum_{k=0}^N |Z_k|^2}$$ is one component of the $SU(N+1)$ moment map $\Psi : \bcp^N \hookrightarrow \mathfrak{su}_{N+1}^*$. The main step of the proof in~\cite{BLY}, is to show that, given a full embedding $\Phi:M \hookrightarrow \bcp^N$, there exists a unique $B\in SL(N+1,\bC)$ such that $B^*=B >0$ and \begin{equation}\label{BalancedCONDITION}\frac{1}{\int_M \omega^n}\int_M (\Psi_{ij}\circ B \circ \Phi)(p) \omega^n = \frac{\delta_{ij}}{N+1}.\end{equation} Said differently,  $(B \circ \Phi)^*\omega_{FS}$ is $(\omega^n/n!)$--balanced, (see~\cite{don:numerics}). To simplify the notation we write $\omega$--balanced instead of $(\omega^n/n!)$--balanced. 

Denote $f_{ij}^B =\Psi_{ij}\circ B \circ \Phi - \frac{\delta_{ij}}{N+1} \, \in C^{\infty}(M)$. The Rayleigh principle implies that \begin{equation}\label{INEQij}
 \lambda_1(M,\omega) \int_M |f_{ij}^B|^2 \frac{\omega^n}{n!} \leq \int_M |\nabla^\omega f_{ij}^B|^2 \frac{\omega^n}{n!}
\end{equation} with equality if and only $f_{ij}^B$ is an eigenfunction of $\Lap$ for the eigenvalue $ \lambda_1(M,\omega)$.

 Taking the sum over $i,j=0,\dots, N$ the left hand side of \eqref{INEQij} gives $$\lambda_1(M,\omega) \frac{N}{N+1} \int_M \frac{\omega^n}{n!}$$ thanks to~\eqref{BalancedCONDITION}. Noticing that $g(\nabla f,\nabla \overline{f})\omega^n = n  \mbox{Re}(df\wedge d^c\overline{f} \wedge\omega^{n-1})$, that the form $\sum_{i,j} d f^B_{ij}\wedge d^c\overline{f^B}_{ij}$ is real and that, on $\bcp^N$,  \begin{equation}\label{PSI=FS}
\sum_{i,j} d \Psi_{ij}\wedge d^c\overline{\Psi}_{ij} = 2\omega_{FS},  
 \end{equation} the right hand side of \eqref{INEQij} gives 
\begin{equation*} 
\sum_{i,j}\int_M |\nabla^\omega f_{ij}^B|^2 \frac{\omega^n}{n!} = \sum_{i,j}\int_M d f^B_{ij}\wedge d^c\overline{f^B_{ij}}\wedge \frac{\omega^{n-1}}{(n-1)!} = \frac{2}{(n-1)!}\int_M (B \circ \Phi)^*\omega_{FS}\wedge \omega^{n-1}. 
\end{equation*} Finally, $B^*\omega_{FS}$ and $\omega_{FS}$ are in the same cohomology class on $\bcp^N$, this concludes the proof. 
\end{proof}
\begin{rem}\label{remEGUALBLY=Eigen}
The equality case in~\eqref{BOUNDbly} implies the equality case in each inequality \eqref{INEQij} and then that each function $f_{ij}^B$ is an eigenfunction of $\Lap$ for the first eigenvalue.   
\end{rem}

In the toric context and for the embedding $\Phi_u$ we get the following refinement of Bourguignon--Li--Yau's result on the existence of balanced metrics.
\begin{lemma} \label{diagonalB}Let $(M,\omega, g_u,J_u, T)$ be a toric K\"ahler manifold with integral polytope $\pol \subset \kt^*$ and corresponding embedding $\Phi_u : M\hookrightarrow \bcp^N$. There exists a diagonal matrix $B=\emph{diag}(\b_0,\dots, \b_N)\in GL(N+1,\bR)$ with $\b_i>0$ and $\emph{tr}B=1$ satisfying the condition~\eqref{BalancedCONDITION}.   
 \end{lemma}

 \begin{proof}
  First, observe that when $i\neq j$, the function $\Psi_{ij}\circ \Phi_u$ integrates to $0$ on $M$ since it does on each orbit of $T$. Hence to prove the lemma, we only have to prove that there exists $\b=(\b_0,\dots, \b_N)\in \bR^{N+1}_{>0}$ such that, for $i=0,\dots, N$, $$\psi_{u,i}(\b) := \b_i^2\int_M \frac{|Z_i|^2}{\sum_{j=0}^N\b_j^2|Z_j|^2}\frac{\omega^n}{n!} = \frac{1}{N+1}$$ where $Z_i = \Phi_{u,m_i}(x,\theta)$ see~\eqref{toricEMBEDDING0}. Let $\Sigma$ be the simplex defined by
$$
  \Sigma :=\left\{ X\in \bR^{N+1}\,|\,\sum_{i=0}^N X_i = 1,\, X_i >0 \right\}.
$$
Because $M$ is full in $\bcp^N$, $\sum_{j=0}^N\b_j^2|Z_j|^2$ does not vanish identically on $M$ (otherwise the $M$ would be contained in a positive codimension subvariety). Since $\sum_{i=1}^N\psi_{u,i}(\alpha) = \mbox{vol}(M)<+\infty$ and each component $\psi_{u,i}(\alpha)\geq 0$ for each $\alpha \in \bR^{N+1}$, by the dominated convergence lemma, the maps $\psi_{u,i}$ can be extended continuously to $\overline{\Sigma}$. Hence, we see $\psi_{u}= (\psi_{u,0}, \dots , \psi_{u,N})$ as a continuous map $$\psi_u : \overline{\Sigma} \longrightarrow \overline{\Sigma}$$ from the closed simplex $\overline{\Sigma}$ to itself. It is obvious that $\psi_{u}$ maps $\del \Sigma$ to $\del \Sigma$. 
  
 To prove the lemma we need to prove that $\psi_u$ is surjective which will follow if we prove that the restriction $\psi_{u}: \del \Sigma \ra \del \Sigma$ has non-trivial degree. 
  
Like in \cite{BLY}, instead of integrating on $M$, we integrate on $\bcp^N$ with the measure $d\mu_u$ defined to be the push forward of the measure on $M$ defined by the metric. This is possible because $\Phi_u$ is a full embedding. Hence $$\psi_{u,i}(\b) = \b_i^2\int_{\bcp^N} \frac{|Z_i|^2}{\sum_{j=0}^N\b_j^2|Z_j|^2}d\mu_u.$$ Now if we consider the volume form $d\mu_o$ induced by the Fubini-Study metric on $\bcp^N$, the corresponding map $\psi_{o}: \del \Sigma \ra \del \Sigma$ with components $$\psi_{o,i}(\b) = \b_i^2\int_{\bcp^N} \frac{|Z_i|^2}{\sum_{j=0}^N\b_j^2|Z_j|^2}d\mu_o$$ and its restriction $\psi_{o}: \del \Sigma \ra \del \Sigma$ are bijections.  
Now $\psi_t = t\psi_u +(1-t)\psi_o$ is a family of continuous maps from $\overline{\Sigma}$ to itself preserving the boundary. The degree of the $\psi_t$ does not depend on $t$ and is non trivial for $t=0$. 
 \end{proof}
 
  \begin{rem} A straightforward corollary of this lemma is that the $\omega$--balanced metric is toric as soon as $\omega$ is toric. \end{rem}

Theorem~\ref{saturate} is a consequence of two propositions we state below and which we prove using the following observation. Assume that the first eigenvalue of $(M,\omega,g_u, J_u, T)$ reaches the Bourguignon, Li and Yau's bound, i.e $\lambda_1(g_u)= \frac{2n(N+1)}{N}$. By Remark~\ref{remEGUALBLY=Eigen} and Lemma~\ref{diagonalB}, there exists a set of $N+1$ real positive numbers $\{\b_k\}_{k\in \pol\cap \bZ^n}$ such that for each $m,k \in \pol\cap \bZ^n$, the function $\Psi_{mk} - \delta_{mk}/(N+1)$ is an eigenfunction of eigenvalue $\frac{2n(N+1)}{N}$ where $$\Psi_{mk}= \frac{\b_k\b_mZ_m\overline{Z}_k}{\sum_{j} |\b_j Z_j|^2}$$ is seen as a function of $(x,\theta)$. Here we write $$Z_m=e^{(u_x +{\bi} \theta)\cdot m},$$ where $u_x=\frac{\del{u}}{\del x}$. We assume, without loss of generality, that $0\in \pol\cap \bZ^n $.  Hence the image of $\Phi_u$, see~\eqref{toricEMBEDDING0}, lies in the set where $Z_0\ne 0$ and we work on this set. We normalize the $\b$'s so as to have $\b_0=1$ instead of $\sum_
{k\in \pol\cap \bZ^n} \b_k =1$ as in the previous lemma. We 
recall that 
\begin{equation}\label{intCONSTRAINT} \int_M \Psi_{mk} \frac{\omega^n}{n!} = \delta_{mk}\int_M \frac{\omega^n}{n!}.\end{equation}

 Observe that for each pair $m,k \in \pol\cap \bZ^n$ 
\begin{equation}\label{RosaTrick1}
 \Lap \Psi_{mk} = \Lap(\b_mZ_m \Psi_{0k}) = \b_mZ_m \Lap\Psi_{0k} - 2 \b_m \langle dZ_m , d\Psi_{0k} \rangle  
\end{equation} since $ \Lap(Z_m)=0$  where $\langle\cdot,\cdot\rangle$ denotes the inner product induced by $g_u$ on the cotangent bundle of $M$. 

When $k=0\neq m$, identity \eqref{RosaTrick1} becomes
\begin{equation*}\begin{split}
\frac{2n(N+1)}{N} \Psi_{m0} &= \frac{2n(N+1)}{N}\b_mZ_m (\Psi_{00} -\frac{1}{N+1}) - 2 \b_m \langle dZ_m , d\Psi_{00} \rangle\\
&= \frac{2n(N+1)}{N}\Psi_{m0} -\frac{2n\b_mZ_m}{N} - 2 \b_m \langle dZ_m , d\Psi_{00} \rangle,
\end{split} 
\end{equation*} and we get 
 \begin{equation*}\frac{2n\b_mZ_m}{N}  = - 2 \b_m \langle dZ_m , d\Psi_{00} \rangle.
\end{equation*} Developing the right hand side in action angle coordinates, using~\eqref{ActionAnglemetric}, we have  
\begin{equation*}\begin{split}\frac{2n\b_mZ_m}{N}  &= - 2 \b_m \langle dZ_m , d\Psi_{00} \rangle\\
&= - 2\b_m\sum_{i,j=1}^n H_{ij}\del_{x_i}Z_m\del_{x_j}\Psi_{00}\\
%&=  \frac{-2\b_m Z_m \sum_{i,j,s=1}^nH_{ij}u_{is}m_s}{\left(\sum_{k\in W} |\b_kZ_k|^2\right)^2} \left(-\sum_{k\in W, t=1,\dots,n} 2u_{tj}k_t|\b_kZ_k|^2\right)\\
&= \frac{4\b_m Z_m \sum_{k\in W}\sum_{s,t=1}^n u_{ts}m_sk_t |\b_kZ_k|^2}{\left(\sum_{k\in W} |\b_kZ_k|^2\right)^2}
\end{split} 
\end{equation*} where $W=\pol\cap \bZ^n$. Dividing both sides by $2\b_m Z_m$, we end up with 

\begin{equation}\label{RosaTrick2} \frac{n}{N} = \frac{2 \sum_{k\in W}\sum_{s,t=1}^n u_{ts}m_sk_t |\b_kZ_k|^2}{\left(\sum_{k\in W} |\b_kZ_k|^2\right)^2}= -(d\Psi_{00})(m) 
\end{equation}

\begin{proposition}\label{CPnsaturates}
Let $(\omega, g_u,J_u, T)$ be a toric K\"ahler structure on $\bC\bP^n$. Assume that $\lambda_1=2(n+1)$ i.e. assume that the first eigenvalue of the Laplacian reaches the Bourguignon--Li--Yau bound.  Then, the toric K\"ahler metric $g_u$ is the Fubini-Study metric on $\bC\bP^n$.
\end{proposition}
\begin{proof}
The moment polytope $P$ of $\bC\bP^n$ is a simplex and if $(\omega, g_u,J_u, T)$ saturates the bound then $\pol$ is primitive as explained in Remark~\ref{rem:optimalCALSS}. So one can suitably normalize it so that it has integer vertices $0, e_1, \cdots e_n$ where $e_i$ is the vector in $\bR^n$ whose $i$-th component is $1$ and all others are zero. In the notation above, $N=n$ and we identify $W=\pol\cap \bZ^n$ with $\{0,1,\dots,n\}$. Equation \eqref{RosaTrick2} implies that, for all $m=1,\dots, n$, 
\begin{equation}\label{eq1} 1 = \frac{2\sum_{m,k=1}^n u_{mk}|\b_kZ_k|^2}{\left(1 + \sum_{k=1}^n |\b_kZ_k|^2\right)^2}= -\frac{\del \Psi_{00}}{\del x_m} 
\end{equation}
where $Z_k=e^{(u_k +{\bi} \theta_k)}$ for $k=1,\dots, n$. Hence  $$
\Psi_{00}= K-\sum_{k=1}^nx_k
$$
for some constant $K$. The additive constant is fixed to be $K=1$ by the integration constraint~\eqref{intCONSTRAINT}.

 In the case $m= k\neq0$, identity \eqref{RosaTrick1} gives $$2(n+1) (\Psi_{mm} -1/n+1)= \b_mZ_m2(n+1)\Psi_{0m} - 2 \b_m \langle dZ_m , d\Psi_{0m} \rangle $$ that is 
\begin{equation*}
-2= -2\b_m \langle dZ_m , d\Psi_{0m} \rangle.    
\end{equation*} Developing the right hand side using~\eqref{eq1}, we get $$-2=-2\b_m \langle dZ_m , d\Psi_{0m} \rangle = 2 |\b_mZ_m|^2\left(\frac{2u_{mm}}{\sum_{i} |\b_i Z_i|^2} -1\right)= - 2\frac{\del \Psi_{mm}}{\del x_m} .$$ Therefore, for each $m>0$ 
\begin{equation}\label{eq_PSIm1}
  \frac{\del \Psi_{mm}}{\del x_m}  =1
\end{equation}
 
 In the case $0\neq m\neq k \neq 0$, identity \eqref{RosaTrick1} gives $2(n+1) \Psi_{mk}= \b_mZ_m2(n+1)\Psi_{0k} - 2 \b_m \langle dZ_m , d\Psi_{0k} \rangle,$ that is 
\begin{equation*}
0= -2\b_m \langle dZ_m , d\Psi_{0k} \rangle.    
\end{equation*}  Developing the right hand side using~\eqref{eq1}, we get 
\begin{equation*}\begin{split}
0 &=-2\b_m \langle dZ_m , d\Psi_{0k} \rangle = 2 \b_m\b_kZ_m\overline{Z}_k\left(\frac{2u_{mk}}{\sum_{i} |\b_i Z_i|^2} -1\right)\\
&= - \frac{2\b_m\b_kZ_m\overline{Z}_k}{|\b_mZ_m|^2}\left( \frac{\del\Psi_{mm}}{\del x_k} \right).
\end{split}\end{equation*} Therefore, for each $m,k>0$ and $k\neq m$, we have  
\begin{equation*}
  \frac{\del \Psi_{mm}}{\del x_k} = 0.
\end{equation*} Together with \eqref{eq_PSIm1} it gives \begin{equation}\label{psi=eigen}\Psi_{mm} = x_m\end{equation} where again the additive constant is fixed by the integration constraint~\eqref{intCONSTRAINT} on $\Psi_{mm}$. We conclude from Remark~\eqref{remEGUALBLY=Eigen} that the components of the moment map are eigenfunction for the same eigenvalue and thus, by Proposition~\ref{propo:moment_map_eigenfunction} $(g_u, J_u)$ is K\"ahler-Einstein. By uniqueness of extremal toric metric~\cite{guan}, $(g_u, J_u)$ is the Fubini-Study metric. \end{proof}

We can also prove that if a toric manifold admits a toric K\"ahler metric for which the embedding given by the integral points of the moment polytope saturates the Bourguignon--Li--Yau bound, then that manifold must be $\bC\bP^n$ and it follows from the above proposition that the metric must be the Fubini-Study metric. 
\begin{proposition}
 Let $(M,\omega,g_u,J_u)$ be a toric K\"ahler manifold with integral polytope $P$. Assume that $\lambda_1(g_u)=\frac{2n(N+1)}{N}$ with $N=\sharp (P\cap \bZ^n)-1$. Then $P$ is the standard simplex and $M$ is (equivariantly symplectomorphic to) $\bC\bP^n$.
 \end{proposition}
\begin{proof} Again we assume that the origin lies in $\pol\cap \bZ^n$, more precisely, up to an integral invertible affine transformation, we may assume that $P$ is standard at the origin i.e. the facets that meet at $0$ have normals $e_1,\cdots, e_n$. In particular, the vertices of $P$ adjacent to the origin, say $m_1,\cdots m_n$, are each an integral multiple of an element of a dual basis of $e_1,\cdots, e_n$ respectively. 

Under the hypothesis of the proposition and with respect to the notation above, Equations \eqref{RosaTrick1} and \eqref{RosaTrick2} hold for points in $\pol\cap \bZ^n$. Suppose there is $m\in P\cap\bZ^n$ a vertex distinct from the origin and from $m_1,\cdots m_n$. Then, there exist $a_1,\cdots a_n$ such that $m=\sum_{l=1}^na_lm_l$.  

Equation \eqref{RosaTrick2} holds for $m$ as well as for $m_1, \cdots m_n$. So we must have 
$$\frac{n}{N}=-(d\Psi_{00})(m) = - \sum_{l=1}^na_l (d\Psi_{00})(m_l) = (\sum_{l=1}^na_l){\frac{n}{N}}$$
which implies $\sum_{l=1}^na_l=1$. So $m$ lies on a facet of the simplex of vertices $0,m_1,\cdots m_n$ which contradicts convexity unless $P$ is that simplex. 
\end{proof}
To sum up we proved Theorem \ref{saturate}.
\bibliographystyle{abbrv}

\end{document}